\documentclass[10pt, a4paper]{amsart}
\usepackage[utf8]{inputenc}
\usepackage{amsfonts}
\usepackage{amssymb}
\usepackage{amsmath}
\usepackage{newlfont}
\usepackage[mathcal]{euscript}
\usepackage{tikz}
\usepackage{enumerate}
\usepackage{hyperref}
\usepackage{amsthm}
\usepackage{cancel}
\usetikzlibrary{decorations.pathreplacing}
\usetikzlibrary{patterns}

\usetikzlibrary{arrows,shapes,positioning}
\usetikzlibrary{decorations.markings}
\tikzstyle arrowstyle=[scale=1]
\tikzstyle directed=[postaction={decorate,decoration={markings,
    mark=at position .65 with {\arrow[arrowstyle]{stealth}}}}]
\tikzstyle reverse directed=[postaction={decorate,decoration={markings,
    mark=at position .65 with {\arrowreversed[arrowstyle]{stealth};}}}]

 \newtheorem{question}{Question}

 \newtheorem{thm}{Theorem}[subsection]
\newtheorem{defn}[thm]{Definition}
\newtheorem{lem}[thm]{Lemma}
\newtheorem{cor}[thm]{Corollary}

\newtheorem{RemI}{Remark}[section]

\newtheorem*{remark*}{Remark}
\newtheorem{definition}[thm]{Definition}

\newtheorem{proposition}[thm]{Proposition}

\newtheorem*{lemmaA}{Lemma A.1}

\DeclareMathOperator{\Int}{int}

\newcommand{\supp}{\mathrm{supp}}
\newcommand{\per}{\mathrm{Per}}
\newcommand{\fix}{\mathrm{Fix}}

 \newcommand{\eps}{\varepsilon}

 \newcommand{\N}{\mathbb{N}}

\numberwithin{equation}{section}

\begin{document}

\title[Lebesgue measure-preserving maps]{Continuous Lebesgue measure-preserving maps on one-dimensional manifolds: a survey}

\author{Jozef Bobok}
\author{Jernej \v Cin\v c}
\author{Piotr Oprocha}
\author{Serge Troubetzkoy}

\address[J.\ Bobok]{Department of Mathematics of FCE, Czech Technical University in Prague, 
Th\'akurova 7, 166 29 Prague 6, Czech Republic}
\email{jozef.bobok@cvut.cz}

\date{\today}

\subjclass[2020]{37E05, 37A05, 37B65, 37C20}
\keywords{Interval map, circle map, Lebesgue measure-preserving, periodic points, generic properties}

\address[J.\ \v{C}in\v{c}]{
Centre of Excellence IT4Innovations - Institute for Research and Applications of Fuzzy Modeling, University of Ostrava, 30. dubna 22, 701 03 Ostrava 1, Czech Republic}
\email{jernej.cinc@osu.cz}

\address[P.\ Oprocha]{
Centre of Excellence IT4Innovations - Institute for Research and Applications of Fuzzy Modeling, University of Ostrava, 30. dubna 22, 701 03 Ostrava 1, Czech Republic}
\email{piotr.oprocha@osu.cz}

\address[S.\ Troubetzkoy]{Aix Marseille Univ, CNRS, Centrale Marseille, I2M, Marseille, France
postal address: I2M, Luminy, Case 907, F-13288 Marseille Cedex 9, France}
\email{serge.troubetzkoy@univ-amu.fr}

\begin{abstract}
We survey the current state-of-the-art about the dynamical behavior of continuous Lebesgue measure-preserving maps on one-dimensional manifolds.
\end{abstract}
\maketitle

\section{Introduction}

Let $M$ denote a compact connected one-dimensional manifold, namely the {\em unit interval}, 
$I :=[0,1]$, and the {\em unit circle}, $\mathbb{S}^1$. Define $C(M)$ to be the set of continuous maps of $M$. Let $\lambda$ denote the 
{\em Lebesgue measure} on an underlying manifold. Our survey will focus on discussion of 
topological and dynamical properties in the space $C_{\lambda}(M)\subset C(M)$ of Lebesgue measure-preserving  
continuous maps of $M$ with the metric of uniform convergence; the space $C_{\lambda}(M)$ is a complete metric space, see \cite[Proposition 4]{BCOT1}. Our particular interest is in Lebesgue 
measure-preserving  continuous maps on $M$ which are not necessarily invertible, so the case that is not usually studied in Ergodic 
Theory.

The class $C_{\lambda}(M)$ contains very large spectrum of maps; on one hand nowhere differentiable ones or even without finite or infinite one-sided derivative \cite{B} and on the other hand  many piecewise monotone maps, many  piecewise smooth maps and of course maps $\textrm{id}$ and $1-\textrm{id}$. Furthermore,  $C_{\lambda}(\mathbb{S}^1)$ also contains all circle rotations.

For a general compact manifold $\mathbb{M}$, let $H_{\lambda}(\mathbb{M})$  denote the space of Lebesgue measure-preserving 
homeomorphisms of $\mathbb{M}$, which is again a complete metric space when equipped with the uniform 
metric. In the setting of volume preserving homeomorphisms in dimension 1, the dynamical behavior is simple and thus not of much interest. However, there are some similarities of
the dynamics of higher dimensional homeomorphisms with the one dimensional continuous maps, so we
will mention them throughout the article.
There is a survey book by Alpern and Prasad on the dynamics of generic volume preserving homeomorphism \cite{AP}, thus we only briefly mention some such results for comparison with $C_{\lambda}(M)$.

Our choice of $C_{\lambda}(M)$ (and $H_{\lambda}(\mathbb{M})$) is motivated by the fact that they are one-dimensional versions of volume-preserving maps, or more broadly, conservative dynamical systems; ergodic maps preserving Lebesgue measure are the most fundamental examples of maps having a unique physical measure. 
 Since generic maps in $C_{\lambda}(M)$ are weakly mixing \cite{BT}, the Ergodic Theorem implies that for a generic map in $C_{\lambda}(M)$ the closure of a  typical trajectory has full Lebesgue measure, thus  the statistical properties of typical trajectories can be revealed by physical observations.
 
Or as Karl Petersen says in the introduction of \cite{Pet}:
 \begin{quote}
     Measure-preserving systems model processes in equilibrium by transformations on probability spaces or, more generally, measure spaces. They are the basic objects of study in ergodic theory, a central part of dynamical systems theory. These systems arise from science and technology as well as from mathematics itself, so applications are found in a wide range of areas, such as statistical physics, information theory, celestial mechanics, number theory, population dynamics, economics, and biology.
 \end{quote}

On the other hand, they represent a variety of possible one-dimensional dynamics as highlighted in the following remark which is proven in the interval case in the introduction of \cite{BCOT1}; in the circle case the proof is analogous using Lemma A.1 from the Appendix.

\begin{RemI}\label{rem:CP}Let $f\in C_{\lambda}(M)$, then  (i) and (ii) are equivalent, and (iii) implies (i).
\begin{itemize}
 \item[(i)] $f$ preserves a non-atomic probability measure $\mu$ with $\supp~\mu=M$. 
    \item[(ii)] There exists a homeomorphism $h$ of $M$ such that $h\circ f\circ h^{-1}\in C_{\lambda}(M)$.
 \item[(iii)] $f$ has a dense set of periodic points, i.e., $\overline{\per(f)}=M$.
 \end{itemize}
Furthermore, if $\per(f)\neq\emptyset$, we have (i) implies (iii). Otherwise $f$ is conjugate to an irrational rotation.
\end{RemI}

The main line of research in this area has been describing the size of the set of maps satisfying a certain topological or dynamical  property. These results are discussed in Section~\ref{sec:denseness} and are summarized in the following table.

 \vspace{0.6cm}
 \hspace{-0.6cm}
\begin{centering}
\begin{tabular}
{|l|c|l|}
\hline
			 Property & $C_\lambda{(M)}$ & Result   \\[2mm]
			\hline
   
   Weak mixing & generic & Theorem \ref{thm:weakmixingI} \& \ref{thm:3}\\
    Strong mixing & dense, but of first category & Theorem \ref{thm:nonmix} \\
    Leo & open dense & Theorem \ref{thm:leoopendense} \&  \ref{thm:leoS1}\\
    Shadowing & generic & Corollary \ref{t-pshadow} \& Theorem \ref{thm:main}\\
    Periodic shadowing & generic & Theorem \ref{t-pshadow} \&  Corollary \ref{cor:shadow}\\
    Specification prop. & open dense &Corollary \ref{cor:opendenseI} \& \ref{cor:specificationcircle}\\
    Hausdorff dim. $\Gamma_f$ & 1 &Theorem \ref{thm:SchWin}  \\
    Lower box dim. $\Gamma_f$ & 1 &Theorem \ref{thm:SchWin}  \\
    Upper box dim. $\Gamma_f$ & 2 & Theorem \ref{thm:SchWin} \\
    $\lambda$-a.e. knot point & generic & Theorem \ref{t:1}\\
$\delta$-crookedness & generic & Theorem \ref{thm:UniLimPresLeb} \\
\hline
		\end{tabular}
  \end{centering}
    \vspace{0.5cm}

 While the previous table summarizes the properties which are the same in $C_{\lambda}(I)$ and $C_{\lambda}(\mathbb{S}^1)$ there are some properties that are different or we do not known if they are the same.
In particular, Theorem \ref{thm:main} shows that s-limit shadowing is generic in $C_{\lambda}(\mathbb{S}^1)$, however, we can only prove the s-limit shadowing property and limit shadowing property are dense in $C_{\lambda}(I)$ (see Proposition \ref{p:1}).
We also give a theorem describing the structure and dimensions of the set of periodic points of generic maps from $C_{\lambda}(I)$ and $C_{\lambda}(\mathbb{S}^1)$. Results are identical in these settings (see Theorems~\ref{t8} and \ref{t8'}) except for the circle maps of degree one (see Theorem~\ref{thm:C_p}).

In Section~\ref{sec:smoothness} we state Theorems \ref{thm:transitiveergodic} and \ref{thm:mixingexact} that show that certain topological and metric properties are equivalent for sufficiently smooth maps. In Subsection~\ref{subsec:nowherediff} we state Theorem~\ref{thm:BobokSoukenka} which shows that there exists a nowhere monotone map in $C_{\lambda}(I)$ that has finite topological entropy. This motivates a more general interesting question regarding the connection between nowhere differentiability and infinite topological entropy.

We finish the article with Section~\ref{sec:other} where we give an overview of known related results for spaces of maps on $M$ equipped with smoother topologies.

Not to disturb the flow of reading, we give in the appendix  a proof of Lemma~\ref{lem:rec} which is needed to argue for Remark~\ref{rem:CP}.

\section{Denseness properties}\label{sec:denseness}

First one needs to note that both spaces $C_{\lambda}(I)$ and $C_{\lambda}(\mathbb{S}^1)$ equipped with the metric of uniform convergence are complete. This follows from Proposition 4 of \cite{BCOT1} which is proven for $M=I$, however the analogous proof works also for $\mathbb{S}^1$. 
 We call a dense $G_{\delta}$ set {\em residual} and call a property {\em generic} if it is attained on at least a residual set of the Baire space on which we work. In this section we will study three different notions of denseness. The strongest property that we verify is that a certain dynamical property holds for an open dense set of maps in $C_{\lambda}(M)$. We verify the weaker notion of genericity for many other topological dynamical properties, while for certain properties we verify only the denseness of maps with certain properties in $C_{\lambda}(M)$.

The study of generic properties in dynamical systems was initiated in  the article by Oxtoby and Ulam from 1941 \cite{OxUl} in which they showed that for a finite-dimensional compact manifold with a non-atomic measure that is positive on open sets, the set of ergodic measure-preserving homeomorphisms is a generic set in the strong topology.
Later Halmos in 1944  \cite{Ha44},\cite{Ha44.1} introduced approximation techniques to a purely metric setting. Namely, he studied interval maps which are invertible almost everywhere and preserve the Lebesgue measure. He showed that the generic invertible map is weakly mixing (i.e., has continuous spectrum).
Subsequently, Rohlin in 1948 \cite{Ro48} showed that Halmos' result is optimal in a sense that the set of strongly mixing measure-preserving invertible maps is of the first category in the same underlying space.
It took until 1967 that this line of research was continued when Katok and Stepin \cite{KaSt} introduced the notion of a speed of approximation. One of the most notable applications of their methods is the proof of genericity of ergodicity and weak mixing for certain classes of interval exchange transformations (IETs). More details on the follow-up history of approximation theory can be found in the surveys
 \cite{ChPr}, \cite{BeKwMe} and \cite{Tr}.

Now we restrict to our particular context. The roots for studying  generic properties on 
$C_{\lambda}(I)$ come from the paper \cite{B} and this line of study was continued recently in 
\cite{BT}, \cite{BCOT}, \cite{BCOT1}, \cite{BCOT2}. The first observation we can make about maps 
from $C_{\lambda}(I)$ is that they have dense set of periodic points. This follows directly from 
the Poincaré Recurrence Theorem and the fact that in dynamical systems given by an interval map 
the closures of recurrent points and periodic points coincide \cite{CoHe80}.  Furthermore, except 
for the two exceptional maps $\mathrm{id}$ and $1-\mathrm{id}$, every such map has positive metric 
entropy. In fact,  except for these two exceptional maps every map is non-invertible on a set of positive measure and thus by a well known theorem (see for example \cite[Corollary 4.14.3]{Walters}) has positive metric entropy
and thus positive topological entropy as well.

\subsection{Main tools}

There are two main tools that are used in most of the results from this section. 

Building on the work of Bobok \cite[Lemma 1]{B} the following lemma was proven in \cite[Proposition 12]{BT} for the interval case, for the definition of a leo map we refer the reader to Subsection~\ref{subsec:TopExp}.
 
\begin{lem}\label{PAMdense}
The set of maps that are piecewise affine Markov and leo are dense in $C_{\lambda}(M)$.
\end{lem}

\begin{proof}
The proof in the circle case follows by combining various known results. 
First of all the density of a special collection of  maps in 
$PA_{\lambda}(\mathbb{S}^1)$ was shown in
Lemmas 13 and 14 from \cite{BCOT}, we can assume all of the absolute values of slopes of these maps 
are at least 4.
Then using Lemma 12 from \cite{BCOT2} we can find a dense set of maps in $PA_{\lambda}(\mathbb{S}^1)$
all of whose critical values are distinct and again of whose slopes are at least 4.
Finally using Lemma 12 from \cite{BT} whose proof also works without change for the circle we have that there is a dense set of Markov maps in 
$PA_{\lambda}(\mathbb{S}^1)$.  In this proof it is implicitly left as an exercise to the reader to show that the resulting map is leo, which we do here. Using  the notation of the proof of this lemma in \cite{BT}, let $J:= (p_{i-1},p_{i+1})$ and
$U$ an arbitrary non-empty open interval. Since the slopes of $f$ and $g_1$ are at least 4
we have $\lambda(g_1(A))/\lambda(A) \ge 4/2=2 $ for any non-empty interval which contains at most one critical point, thus
we can find an $n$ such that $V:= g_1^n(V)$ contains at least two consecutive critical points.  From the construction of $g_1$ we have
$f(V \setminus J) = g_1(V \setminus J)$, $f(V \cap J) \subset g_1(V \cap J)$ provided that $\{p_{i-1},p_{i+1}\} \ne \emptyset$, and if $V \subset J$ then since $V$ contains two consecutive critical points we have $g_1(V) = g_1(J) =f(J) \supset f(V)$; and thus $g_1$ is leo.
\end{proof}

The proofs in this section also use window perturbations as the other main tool.
Let $J$ be an arc in $M$ (i.e., a homeomorphic image of $[0,1]$).
 Let $m$ be an odd positive integer and $\{J_i \in M: 1 \le i \le m \}$ a finite collection of arcs satisfying $\cup^{m}_{i=1} J_i = J$ and $\Int(J_i) \cap \Int(J_j) = \emptyset $ when $i \ne j$. We will refer to this as a {\em partition} of $J$.
 
 Fix $f \in C_{\lambda}(M)$
 an arc $J \in M$ and a partition $\{J_i\}$ of $J$.
 A map $h \in C(M)$ is {\em an $m$-fold window perturbation of $f$ with respect to $J$ and the partition $\{J_i\}$} if
 \begin{itemize}
     \item $h|_{J^c} = f|_{J^c}$ 
     \item for each $1 \le i \le m$ the map $h|_{J_i}$ is an affinely scaled copy of $f|_{J}$ with the orientation reversed for every second $i$, with  $h|_{J_1}$ having the same orientation as $f|_{J}$.
 \end{itemize}
 
The essence of this definition is illustrated by Figure~\ref{fig:perturb}.
 
\begin{figure}[!ht]
	\centering
	\begin{tikzpicture}[scale=3.5]
	\draw (0,0)--(0,1)--(1,1)--(1,0)--(0,0);
	\draw[thick] (0,1)--(1/2,0)--(1,1);
	\node at (1/2,1/2) {$f$};
	\node at (5/16,-0.1) {$a$};
	\node at (5/8,-0.1) {$b$};
	\node at (0,-0.1) {$0$};
	\node at (1,-0.1) {$1$};
	\node at (-0.1,1) {$1$};
	\draw[dashed] (5/16,0)--(5/16,3/8)--(5/8,3/8)--(5/8,0);
	\node[circle,fill, inner sep=1] at (5/16,3/8){};
	\node[circle,fill, inner sep=1] at (5/8,1/4){};
	\end{tikzpicture}
	\hspace{1cm}
	\begin{tikzpicture}[scale=3.5]
	\draw (0,0)--(0,1)--(1,1)--(1,0)--(0,0);
	\draw[thick] (0,1)--(5/16,3/8)--(5/16+3/48,0)--(5/16+5/48,1/4)--(5/16+7/48,0)--(5/16+10/48,3/8)--(5/16+13/48,0)--(5/8,1/4)--(1,1);
	\draw[dashed] (5/16,0)--(5/16,3/8);
	\draw[dashed] (5/16+10/48,0)--(5/16+10/48,3/8);
	\draw[dashed] (5/16+5/48,3/8)--(5/16+5/48,0);
	\node at (1/2,1/2) {$h$};
	\node at (0,-0.1) {$0$};
	\node at (1,-0.1) {$1$};
	\node at (-0.1,1) {$1$};
	\node at (5/16,-0.1) {$a$};
	\node at (5/8,-0.1) {$b$};
	\draw[dashed] (5/16,3/8)--(5/8,3/8)--(5/8,0);
	\node[circle,fill, inner sep=1] at (5/16,3/8){};
	\node[circle,fill, inner sep=1] at (5/16+5/48,1/4){};
	\node[circle,fill, inner sep=1] at (5/16+10/48,3/8){};
	\node[circle,fill, inner sep=1] at (5/8,1/4){};
	\end{tikzpicture}
	\caption{For  $f\in C_{\lambda}(M)$ shown on the left, we show on the right the graph of $h$ which is
 a $3$-fold piecewise window perturbation of $f$ on the interval $[a,b]$.}\label{fig:perturb} 
\end{figure}
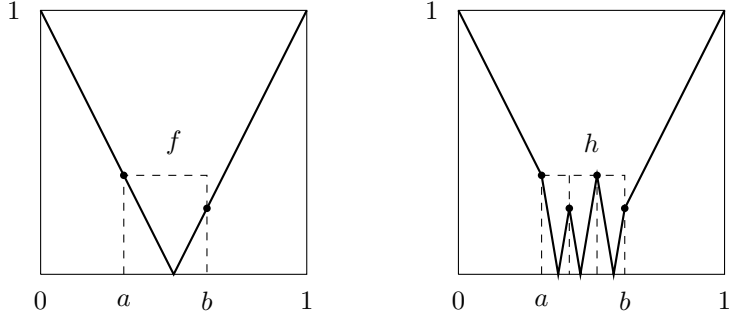

\subsection{Measure-theoretic properties}

Let $\mathcal{B}$ denote the Borel sets in $M$.
The measure-preserving transformation  $(f,M,\mathcal{B},\mu)$ is called
\begin{itemize}
	\item \textit{ergodic} if for every $A\in \mathcal{B}$, $f^{-1}(A)=A$ $\mu$-a.e. implies that $\mu(A)=0$ or $\mu(A^c)=0$.
	\item  \textit{weakly mixing}, if for every $A,B\in\mathcal{B}$,
	$$\lim_{n\to\infty}\frac{1}{n}\sum_{j=0}^{n-1}\vert \mu(f^{-j}(A)\cap B)-\mu(A)\mu(B)\vert=0.$$
    \item \textit{strongly mixing}  if for every $A, B \in \mathcal{B}, \lim_{n\to\infty}\mu(f^{j}(A)\cap B)=\mu(A)\mu(B)$.

	\item {\em measure-theoretically exact} if for each set $A\in \cap_{n\geq 0}f^{-n}(\mathcal{B})$ it holds that $\mu(A)\mu(A^{c})=0$.
\end{itemize}
In this paper we will mainly focus on the case of particular invariant measure, the Lebesgue measure $\lambda$.

The following measure-theoretic properties of $C_{\lambda}(I)$-generic functions were proven in \cite{BT}. 

\begin{thm}\label{thm:weakmixingI}
The $C_{\lambda}(I)$-generic function 
\begin{enumerate}
	\item is weakly mixing with respect to Lebesgue measure $\lambda$ \cite[Th. 15]{BT},
	\item maps a set of Lebesgue measure zero onto $I$ \cite[Cor. 22]{BT}.
\end{enumerate}
\end{thm} 

In analogy to Rohlin's classical result \cite{Ro48}, the following result was proven for the interval in  \cite[\S 3]{BT}. For the circle case we need several modifications, first for the density \cite[Proposition 12]{BT}
is replaced by Lemma \ref{PAMdense}.  Except for the  application of the Block Coven result (\cite[Theorem 9]{BC}, \cite[Theorem 18]{BT}), all the other steps work for the circle without change.  To apply \cite[Theorem 9]{BC}  in the circle case it suffices to start with a sufficiently fine Markov partition such that 
the diameter of the  image of each element of the partition is less than the diameter of the circle.
\begin{thm}\label{thm:nonmix}
The set of mixing maps in $C_{\lambda}(M)$ is dense and  
is of the first category.
\end{thm}
 
Returning to the case when $M=\mathbb{S}^1$,
the following result was obtained in \cite[Theorem 2]{BCOT2}.  For each $\alpha\in[0,1)]$, let $r_{\alpha}\colon \mathbb{S}^1\to \mathbb{S}^1$ be a circle rotation for the angle $\alpha$.
Define the operator $T_{\alpha,\beta}\colon C_{\lambda}(\mathbb{S}^1) \mapsto  C_{\lambda}(\mathbb{S}^1)$ by $T_{\alpha,\beta}(f)=r_{\alpha}\circ f\circ r_{\beta}.$ 
 
 \begin{thm}\label{thm:3}
 	There exists a dense $G_{\delta}$ subset $\mathcal{G}$ of $C_{\lambda}(\mathbb{S}^1)$ such that
 	\begin{enumerate} \item each $g \in \mathcal{G}$ is weakly mixing with respect to $\lambda$, 
  \item each $g \in \mathcal{G}$ maps a set of Lebesgue measure zero onto $\mathbb{S}^1$, and
 		\item
 		for each pair $\alpha,\beta\in [0,1)$ and each $g\in \mathcal{G}$ the map $T_{\alpha,\beta}(g)\in \mathcal{G}$.
 	\end{enumerate}
 \end{thm}

 Point (2) was shown in \cite[Cor. 22]{BT} for the interval, the proof holds without change in the case of the circle, and furthermore if $g$ maps a set of Lebesgue measure zero onto $\mathbb{S}^1$, then so does $T_{\alpha,\beta}(g)$.

 In \cite[Theorem 2]{BCOT1} it was shown that there is a dense set of  non-ergodic maps in $C_{\lambda}(I)$; its proof works with natural modifications also in $C_{\lambda}(\mathbb{S}^1)$. Thus Theorem~\ref{thm:3} is optimal  since there is no nonempty open set of maps satisfying nice mixing properties for any $\alpha,\beta$.

For volume preserving homeomorphisms we have the following classical result on the unit cube $I^n$ due to 
Oxtoby and Ulam  \cite{OxUl} (see also\cite[Theorem 7.1]{AP}):
\begin{thm}\label{thm:OU}
    The generic $f \in H_{\lambda}(I^n)$ is ergodic for all $n \ge 2$.
\end{thm}
\noindent 
The improvement of Theorem~\ref{thm:OU} was shown in \cite[Theorem 1.2]{A00}:
\begin{thm}
The generic $f \in H_{\lambda}(\mathbb{M})$ is weakly mixing and periodic points of $f$ are dense in $\mathbb{M}$ for any compact manifold $\mathbb{M}$ of dimension at least 2.
\end{thm}
 
\subsection{Topological expansion properties}\label{subsec:TopExp}
We call a map $f\in C_{\lambda}(M)$ 
\begin{itemize}
	\item \textit{transitive} if for all nonempty open sets $U,V\subset M$ there exists $n\ge 0$ so that $f^n(U)\cap V\neq\emptyset$,
	\item \textit{totally transitive} if $f^n$ is transitive for all $n\geq1$.
	\item \textit{(topologically) weakly mixing} if its Cartesian product $f\times f$ is transitive,
	\item \textit{topologically mixing} if for all nonempty open sets $U,V\subset M$ there exists $n_0\geq0$ so that $f^n(U)\cap V\neq\emptyset$ for every $n\ge
	n_0$,
	\item  \textit{leo} (\textit{locally eventually onto}, also known as \textit{topologically exact}) if for every nonempty open set $U\subset M$ there exists $n\in{\mathbb N}$ so that $f^n(U)=M$.
\end{itemize}

By the usual hierarchy in topological dynamics,  every leo map is topologically mixing, topologically weakly mixing, totally transitive and transitive.
In \cite[Theorem 9]{BT} the following theorem  was shown:

\begin{thm}\label{t2}
The $C_{\lambda}(I)$-generic function is leo.
\end{thm}

Theorem \ref{t2} was recently improved in \cite{BCOT3} 
where the following result was shown:
 
 \begin{thm}\label{thm:leoopendense}
 There is an open dense set of leo maps  in $C_{\lambda}(I)$.
 \end{thm}

If we take any continuous map of $I$
with attracting fixed point in the interior of $I$ then it is clear that all sufficiently close maps in $C(I)$ cannot be transitive. For any map in $C(I)$ we can blow up an attracting to an invariant set with attracting fixed point. Thus non-transitive continuous interval maps and also continuous interval maps with $\overline{\per(f)}\neq I$ form  open dense sets.

 By the result of Blokh  \cite[Theorem 8.7]{Blokh}, for interval maps topological mixing implies periodic specification property (cf. proof of Buzzi \cite[Appendix A]{Buzzi}). Therefore, we obtain the following corollary.

\begin{cor}\label{cor:opendenseI}
	There is an open dense set of maps from $C_{\lambda}(I)$ satisfying the periodic specification property.
\end{cor}
 
Now we turn to the the circle case, where a more detailed study has been realized recently in \cite{BCOT2}. 
 
 \begin{thm}\label{thm:leoS1}
 	There is an open dense set of maps $\mathcal{O}\subset C_{\lambda}(\mathbb{S}^1)$ such that:
 	\begin{enumerate}
 		\item 
   each map $f\in \mathcal{O}$ is leo.
 		\item
   for each pair of $\alpha,\beta\in [0,1)$ and each $f\in \mathcal{O}$ the map $T_{\alpha,\beta}(f)\in \mathcal{O}$.
 	\end{enumerate}
 \end{thm}

 In fact, Blokh's results about specification property mentioned above generalize to all topologically mixing maps on topological graphs, in particular to the circle (see \cite{BlokhBM1}, cf.~\cite{BlokhSurv,AdRR}; another approach for proving the generalization of Blokh's results on topological graphs can be found in \cite{HKO} and was inspired by the techniques of Buzzi for interval maps \cite[Appendix A]{Buzzi}.

\begin{cor}\label{cor:specificationcircle}
	There is an open dense set of maps from $C_{\lambda}(\mathbb{S}^1)$ satisfying the periodic specification property.
\end{cor}

By the same argument as for the interval maps, this result does not extend to the whole set $C(\mathbb{S}^1)$.

We turn to the case of Lebesgue measure-preserving homeomorphisms.
A map $f:X \to X$ of a compact metric space $X$ is called {\em maximally chaotic} if 
\begin{enumerate}
    \item\label{def:max:1} $f$ is topologically transitive,
    \item\label{def:max:2} the periodic points of $f$ are dense in $X$, and
    \item\label{def:max:3}  $\limsup_{k \to \infty} \text{diam}(f^k (U)) = \text{diam}(X)$ for any non-empty open set $U \subset X$.
\end{enumerate}
Notice that maximal chaos implies the well known Devaney chaos.
While the leo property is impossible for homeomorphisms, we have the following result in the setting of the $n$-dimensional cube $I^n$, which summarizes the results  found in \cite[Theorems 4.5 and D]{AP}:

\begin{thm}
For $n \ge 2$ the generic $f \in H_{\lambda}(I^n)$ is topologically weakly mixing and maximally chaotic.
\end{thm}

It is not hard to see that conditions \eqref{def:max:1}
and \eqref{def:max:3} in the definition of maximal chaos are immediate consequence of weak mixing. Condition \eqref{def:max:2} is not in general consequence of weak mixing, since there exist weakly mixing minimal systems.

\subsection{Periodic points and dimension properties} 
The Hausdorff, lower and upper box dimension of  the graph
of a function is a way to describe the ``roughness'' of the function.
The following theorem by Schmeling and Winkler \cite{SW95} was stated for maps from $C_{\lambda}(I)$, but it holds in any dimension\footnote{Personal communication from Jörg Schmeling.}; in particular, in dimension one we have the following theorem.
\begin{thm}\label{thm:SchWin}
A graph of a generic map $f \in C_{\lambda}(M)$
has Hausdorff dimension equal to the lower box dimension, 
both equal to $1$ and the upper box dimension equal to $2$. 
\end{thm}

Understanding of 
 the structure of the set of  periodic points of the function under consideration is among the fundamental tasks in dynamical systems theory. Since generic maps from 
$C_{\lambda}(I)$ are weakly mixing with respect to $\lambda$ it follows that the 
Lebesgue measure of the set of periodic points is equal to $0$. It is interesting to 
study the finer structure of the set of periodic points for generic maps, in 
particular its cardinality and dimension.

The set of 
periodic points of period $k$ for $f$ is denoted $\per(f,k)$, the set of fixed points of $f^k$ is denoted by $\fix(f,k)$ and of the union of all periodic points of $f$ is denoted by $\per(f)$.
In \cite{BCOT1} the authors studied the cardinality and structure of the set of periodic points and its respective lower box, upper box and Hausdorff dimensions:

\begin{thm}\label{t8}
	For a generic map $f \in C_{\lambda}(I)$,  for every $k \geq 1$:
	\begin{enumerate}
		\item 
  $\fix(f,k)$ is a Cantor set, 
		\item  
  $\per(f,k)$ is a relatively open dense subset of the set $\fix(f,k)$, 
		\item 
  the set $\fix(f,k)$ has lower box dimension and Hausdorff dimension  zero. In particular, $\per(f,k)$ has lower box dimension and Hausdorff dimension zero. As a consequence, the Hausdorff dimension of $\per(f)$ is also zero.
		\item  
  the set $\per(f,k)$ has upper box dimension one. Therefore, $\fix(f,k)$ has upper box dimension one as well.
	\end{enumerate}
\end{thm}

In the setting of $C_{\lambda}(\mathbb{S}^1)$, due to the presence of rotations, degree $1$ maps need to be treated separately. Denote the set of degree $d$ maps in $C_{\lambda}(\mathbb{S}^1)$
by $C_{\lambda,d}(\mathbb{S}^1)$.

The proof of Theorem \ref{t8} shows:

\begin{thm}\label{t8'}
Conclusions of Theorem \ref{t8}
hold also for generic maps in $C_{\lambda,d}(\mathbb{S}^1)$ for each $d \in \mathbb{Z} \setminus \{1\}$.
\end{thm}

The proofs of these two theorems  can easily be adapted to show that the generic map in $C(I)$ and degree $d$ maps in $C(\mathbb{S}^1)$ (i.e., not necessarily measure preserving) have the same properties (see \cite{BCOT1}[Remark 14 and 18]). 

 For $C_{\lambda,1}(\mathbb{S}^1)$ the situation is more complicated. 
A periodic point $x \in \per(f,k) $ is called \emph{transverse}, if the graph of $f^k$ crosses the diagonal at $x$ (possibly coincides with the diagonal on an interval containing $x$).
Consider the open set
$$C_p :=\{f \in C_{\lambda,1}(\mathbb{S}^1): f \text{ has a transverse periodic point of period } p\}.$$
In this setting the proof of  Theorem  \ref{t8} yields a similar result from \cite{BCOT1}. 

\begin{thm}\label{thm:C_p}
For any $f$ in a dense $G_\delta$ subset of $\overline{C}_p$ we have for each
$k \in \N$
\begin{enumerate}
\item $\fix(f,kp)$ is a Cantor set, 
\item $\per(f,kp)$ is a relatively open dense subset of $\fix(f,kp)$, 
\item $\fix(f,kp)$ has lower box dimension and Hausdorff dimension zero. In particular, $\per(f,kp)$ has lower box dimension and Hausdorff dimension zero.  As a consequence, the Hausdorff dimension of $\per(f)$ is zero as well.
\item  the set $\per(f,kp)$ has upper box dimension $1$. Thus, $\fix(f,kp)$ and $\per(f)$ both have upper box dimension also $1$.
\end{enumerate}
\end{thm}

To interpret this result  the set  $C_{\infty} := C_{\lambda,1}(\mathbb{S}^1) \setminus \cup_{p \ge 1} \overline{C_p}$ was studied in \cite{BCOT1}. It turns out that a periodic point can
be transformed to a transverse periodic point by an arbitrarily small perturbation of the map, thus
the set $C_{\infty}$  consists of maps without periodic points. Using the same argument one can see that $\cup_{p \ge 1} \overline{C_p}$ contains an open dense set. Therefore, $C_{\infty}$ is nowhere dense in $C_{\lambda,1}(\mathbb{S}^1)$. Furthermore, in \cite{BCOT1} the following complete characterization of the set $C_{\infty}$ was obtained: 

\begin{proposition}
The set $C_{\infty}$ consists of irrational circle rotations.
\end{proposition}

Periodic points of generic homeomorphisms have been well studied also in other contexts.
Akin et.~al. proved that the set of periodic points is a Cantor set for generic homeomorphisms of $\mathbb{S}^1$ \cite[Theorems 9.1 and 9.2(a)]{AHK}. 
On the other hand, Carvalho et.~al.~have shown that the upper box dimension of the set of periodic points  is full (i.e., of the same dimension as the dimension of the underlying manifold) for generic homeomorphisms on compact manifolds of dimension at least one  \cite{PV}\footnote{this statement only appears in the published version of \cite{PV}.}.

Let us recall that generic maps from $C_{\lambda}(I)$ will necessarily have Lebesgue measure $0$ on the set of periodic points since $\lambda$ is weakly mixing. Nonetheless the following somewhat surprising result is proven in \cite[Theorem 2]{BCOT1}.

\begin{thm}\label{t-PP}
	The set of leo maps in $C_{\lambda}(I)$ whose periodic points have full Lebesgue measure and whose periodic points of period $k$ have positive Lebesgue measure for each $k \ge 1$
	is dense in the set $C_{\lambda}(I)$.
\end{thm} 

The following result about volume preserving homeomorphisms is proven in an unpublished sketch by
Guihéneuf \cite{PAG}.

\begin{thm}
The set of periodic points of a generic $f \in H_{\lambda}(\mathbb{M})$
for a compact manifold $\mathbb{M}$ of dimension at least two 
is a dense set of zero measure and for every $\ell \ge 1$ the set of 
fixed points of $f^{\ell}$  is either empty or perfect. 
\end{thm}
More generally Guihéneuf's result holds for homeomorphisms preserving a ``good''
measure in the sense of Oxtoby and Ulam \cite{OxUl}.

\subsection{Shadowing properties}
One of the classical notions from topological dynamics is the so-called shadowing property. It is of particular importance in systems possessing sensitive dependence on initial conditions. In such systems, very small errors could potentially lead to large divergence of orbits. Shadowing is a notion arising from computer science and is used as a tool for determining if any hypothetical orbit is indeed close to some real orbit of a topological dynamical system.
It assures that the dynamics of maps which satisfy it can be realistically observed through computer simulations.
 Let us first give definitions that are important for this subsection.
 
 For $\delta> 0$, we call a sequence $(x_n)_{n\in \N_0}\subset I$ a \emph{$\delta$-pseudo orbit} of $f\in C(I)$ if $d(f(x_n), x_{n+1})< \delta$ for every $n\in \N_0$. A \emph{periodic $\delta$-pseudo orbit} is a $\delta$-pseudo orbit for which there is $N\in\N_0$ so that $x_{n+N}=x_n$, for every $n\in \N_0$. The sequence $(x_n)_{n\in\N_0}$ is called an \emph{asymptotic pseudo orbit} if $\lim_{n\to\infty} d(f(x_n),x_{n+1})=0$.
 Provided a sequence $(x_n)_{n\in\N_0}$ is a $\delta$-pseudo orbit and an asymptotic pseudo orbit we say it is an asymptotic $\delta$-pseudo orbit.

\begin{definition}
We say that a map $f\in C(M)$ has the:
\begin{itemize}
 \item \emph{shadowing property} if for every $\eps > 0$ there exists $\delta >0$ satisfying the following: given any $\delta$-pseudo orbit $\mathbf{y}:=(y_n)_{n\in \N_0}$ we can find a corresponding point $x\in M$ that $\eps$-traces $\mathbf{y}$, i.e.,
$$d(f^n(x), y_n)<  \eps \text{ for every } n\in \N_0.$$
\item
\emph{periodic shadowing property} if for every $\eps>0$ there exists $\delta>0$ satisfying the following condition: given any periodic $\delta$-pseudo orbit $\mathbf{y}:=(y_n)_{n\in\N_0}$ we can find a corresponding periodic point $x \in M$, which $\eps$-traces $\mathbf{y}$.
\item \emph{limit shadowing property} if for every \emph{asymptotic pseudo-orbit}, i.e. sequence $(x_n)_{n\in \N_0}\subset M$, so that $$d(f(x_n),x_{n+1})\to 0 \text{ when } n\to \infty$$
there exists $p\in M$, so that
$$d(f^n(p),x_n)\to 0 \text{ as } n\to \infty.$$
\item \emph{s-limit shadowing property} if for every $\eps>0$ there exists $\delta>0$ so that
\begin{enumerate}
\item  for every $\delta$-pseudo orbit $\mathbf{y}:=(y_n)_{n\in \N_0}$ we can find a corresponding point $x\in M$ which $\eps$-traces $\mathbf{y}$,
\item  for every asymptotic $\delta$-pseudo orbit $\mathbf{y}:=(y_n)_{n\in \N_0}$ of $f$, there exists $x\in M$ which $\eps$-traces $\mathbf{y}$ and
$$ \lim_{n\to \infty}d(y_n,f^n(x)) = 0.$$
\end{enumerate}
\end{itemize}
\end{definition}

The following theorem was proved by the authors in \cite[Theorem 3]{BCOT1}.

\begin{thm}\label{t-pshadow}
	The shadowing and periodic shadowing properties are generic for maps from $C_{\lambda}(I)$.
\end{thm}

 For comparison, in the larger space $C(M)$ Mizera proved in \cite{Mizera} that the shadowing property is  generic.  
 Several other results that shadowing is generic in topology of uniform convergence in more general settings were established  (see \cite{Med, KoMaOpKu}) using the techniques of  Pilyugin and Plamenevskaya \cite{Pi} initially developed for proving genericity of shadowing property of homeomorphisms on any smooth compact manifolds without a boundary.

\begin{proposition}\label{p:1}
	The set of maps $f\in C_{\lambda}(I)$ that have s-limit shadowing property is dense in the set $C_{\lambda}(I)$.
\end{proposition}

The main result in \cite{BCOT} for the setting $C_{\lambda}(\mathbb{S}^1)$ is even stronger than the preceding two statements. 

\begin{thm}\label{thm:main}
	The s-limit shadowing property is generic in $C_{\lambda}(\mathbb{S}^1)$.
\end{thm}

\begin{cor}\label{cor:shadow}
	The limit shadowing, periodic shadowing and shadowing property are generic in $C_{\lambda}(\mathbb{S}^1)$.
\end{cor}

Actually, Theorem~\ref{thm:main} is somewhat surprising since $C_{\lambda}(\mathbb{S}^1)$ is the first environment in which s-limit shadowing was proven to be generic. Up to now, only denseness of s-limit shadowing was established in the setting
of compact topological manifolds
\cite{MazOpr}. Theorem \ref{thm:main} also holds in the setting of  $C(\mathbb{S}^1)$. However, the methods used in \cite{BCOT} do not work in the setting of $C_{\lambda}(I)$. This motivates the following.

\begin{question}
Is s-limit shadowing generic in $C_{\lambda}(I)$?
\end{question}

 Possible positive answer to the above question will require some new techniques than the ones used in \cite{BCOT}.
On the other hand, a standard technique to disprove that a condition is generic is to find an open set without the required property. Such approach is again impossible, because of Proposition~\ref{p:1}. In the view of Theorem~\ref{thm:main} and the result from \cite{GuLe18} it is also natural to ask the following question.

\begin{question}
Is s-limit shadowing generic also for the volume preserving homeomorphisms on manifolds of dimension greater than $1$?
\end{question}

 In the context of volume preserving homeomorphisms on manifolds of dimension at least two (with or without boundary), the genericity of shadowing was recently proven by Guihéneuf and Lefeuvre \cite{GuLe18}.

\subsection{Knot points}\label{subsec:knotpoint} 
We define the upper, lower, left and right \textit{Dini derivatives}
of $f$ at $x$:
\begin{align*}
D ^{+}f(x)& :=  \limsup_{\scriptstyle t \to x^{+}}{\frac {f(t)-f(x)}{t-x}} \qquad
D _{+}f(x) :=  \liminf_{\scriptstyle t\to x^{+}}{\frac {f(t)-f(x)}{t-x}}\\
D ^{-}f(x)& :=  \limsup_{\scriptstyle t\to x^{-}}{\frac {f(t)-f(x)}{t-x}} \qquad
D _{-}f(x) :=  \liminf_{\scriptstyle t\to x^{-}}{\frac {f(t)-f(x)}{t-x}}.
\end{align*}

We call a point $x\in M$ a {\em knot point} of function $f\in C(M)$ if suprema and infima of the right and left derivatives at point $x$ satisfy $D^{+}f(x) = D^{-}f(x) = \infty$ and $D_{+}f(x) = D_{-}f(x) = -\infty$. The following theorem states a consequence of a more general result proved in \cite{B} for the interval, the circle case can be treated analogously.

\begin{thm}\label{t:1}~The $C_{\lambda}(M)$-generic function has a knot point at $\lambda$-almost every point.
\end{thm}
The next result generalizes a classical result of Saks
\cite{Sa32} saying that the set of Besicovitch functions is a meager set in $C(I)$.  Its circle version follows from the fact that a monotonicity result can be applied separately on arcs partitioning the circle.  

A {\it Besicovitch function}  $f\in C(M)$ is a map such that for
every $x\in M$, no unilateral finite or infinite derivative exists at $x$.

\begin{cor}
The set of Besicovitch functions is a meager set in  $C_{\lambda}(M)$.
\end{cor}

\begin{proof}
We use the following well known result (see \cite[Theorem 7.3]{Sa37}): if for an arc $A\subset M$ and $f\colon~A\to M$ we have \textit{$D^+f(x) \ge 0$ for a.e. $x \in A$ and $D^+f(x) > - \infty$
for every $x \in A$, then $f$ is non-decreasing.}

By Theorem \ref{t:1} there is a residual set $K\subset C_{\lambda}(M)$
such that each element
of $K$ has a knot point at $\lambda$ almost every point of $M$. Fix $f
\in K$ and an arc $A\subset M$;  we have
 $D^+f(x) = +\infty \ge 0$ a.e.\ on $A$ hence $f$ can not be non-decreasing.
Applying the above result, we conclude that $D^+(x_0) = - \infty$ for
at least one point $x_0 \in A$;
in particular $f$ is not a Besicovitch function.
\end{proof}

A  {\it Morse function} $f\in C(M)$ satisfies
$$\max\{\vert D^+f(x)\vert,\vert D_+f(x)\vert\}=\max\{\vert
D^-f(x)\vert,\vert D_-f(x)\vert\}=\infty,~x\in M;
$$
in the interval case, if $x$ is an endpoint of $I$, the only max is  taken over
the derivatives from inside $I$. There exists a function $f\in C(M)$ which is Besicovitch and Morse. 
The following question remains open.

\begin{question}
Does there exist a Besicovitch-Morse function in $C_{\lambda}(M)$?
\end{question}

\subsection{Crookedness} 
We begin by defining the 
 crookedness property which is one of the central properties in Continuum Theory. 
 We explain its importance later.

\begin{defn}
	Let $f\in C(I)$,  $a,b\in I$ and let $\delta>0$. We say that $f$ is {\em $\delta$-crooked between $a$ and $b$} if for any two points $c, d \in I$ so that $f(c) = a$ and $f(d) = b$, there is a point $c'$ between $c$ and $d$ and there is a point $d'$ between $c'$ and $d$ so that $|b - f(c')| < \delta$ and $|a - f(d')| < \delta$. We will say that $f$ is {\em $\delta$-crooked} if it is $\delta$-crooked between any pair of points.
\end{defn}

In \cite[Theorem 1]{CO} and \cite{COpseudocircle} the authors proved the following generic property of maps from $C_{\lambda}(M)$, which might be the most surprising of the generic properties proven yet:

\begin{thm}\label{thm:UniLimPresLeb}
	There is  a dense $G_\delta$ set $\mathcal{T}\subset C_{\lambda}(M)$ such that if $f\in \mathcal{T}$ then
	for every $\delta>0$ there exists a positive integer $n$ so that $f^n$ is $(f,\delta)$-crooked. 
\end{thm}

The $\delta$-crookedness condition is a topological condition that imposes strong requirements on values of the map.
Piecewise smooth maps do not verify the crookedness condition, thus Theorem~\ref{thm:UniLimPresLeb} cannot hold for any open collection of maps in $C_{\lambda}(M)$.

The pseudo-arc is a very curious object arising from Continuum Theory (see the survey of Lewis \cite{Lewis} and the introduction of \cite{BCO} for the overview of results involving the pseudo-arc), which was first discovered by Knaster over a century ago. On one hand side, its complicated structure is reflected by the fact that it is {\em hereditarily indecomposable}, i.e., there are no proper subcontinua $A,B\subset H$ such that $A\cup B=H$ for every proper subcontinuum $H$ of the pseudo-arc $P$. On the other hand, the pseudo-arc is {\em homogeneous}, i.e., for every two points $x,y\in P$ there exists a homeomorphism $h:P\to P$ such that $h(x)=y$. Homogeneity is a property possessed by the spaces with locally identical structure. Non-trivial examples of homogeneous spaces are the Cantor set, solenoids and manifolds without boundaries, for instance the circle.

Let $\{Z_i\}_{i\geq 0}$ be a collection of compact metric spaces.
For a collection of continuous maps $f_i:Z_{i+1}\to Z_i$ we define
\begin{equation*}
\underleftarrow{\lim} (Z_i,f_i)
:=
\{\hat z:=\big(z_{0},z_1,\ldots \big) \in Z_0\times Z_1,\ldots\big|  
z_i\in Z_i, z_i=f_i(z_{i+1}), \forall i \geq 0\}.
\end{equation*}
We equip inverse limit $\underleftarrow{\lim} (Z_i,f_i)$ with the subspace 
metric which is induced from the 
\emph{product metric} in $Z_0\times Z_1\times\ldots$,
where $f_i$ are called the {\em bonding maps}.

\begin{cor}\label{cor:MincTransue}
	The inverse limit with any $C_{\lambda}(I)$-generic map as a single bonding map is the pseudo-arc.
\end{cor} 

This corollary is a direct consequence of Theorem~\ref{thm:UniLimPresLeb} and a result of Minc and Transue \cite[Proposition 4]{MT} connecting crookedness with pseudo-arc as inverse limit.

\subsection{Entropy}

The property that the topological entropy of generic maps on the Lebesgue measure-preserving maps is $\infty$ can be deduced from the methods of the article of Yano \cite{Y}. Moreover, another way to see it from general theory is to combine Theorem~\ref{cor:MincTransue} with results of \cite{ChrisM}.
The connection between \cite{Y} and $C_\lambda(I)$ was explicitly done in \cite[Proposition 26]{BT}:

\begin{thm}
The generic value of topological entropy in $C_{\lambda}(I)$ is $\infty$.
\end{thm}
  
 The proof from \cite{BT} easily extends to $\mathbb{S}^1$, by replacing the fixed point in the  proof with a periodic point.  
 The generic value of topological entropy for the volume preserving continuous maps seems not to have been studied for other manifolds. On the other hand Guihéneuf \cite[Théorème 3.17]{PAG1} proved the analogous  result holds for generic homeomorphisms of 
 compact, connected manifold of dimension at least 2 preserving a ``good''
measure in the sense of Oxtoby and Ulam \cite{PAG1}.
The analogous result in the setting of homeomorphisms on manifolds of dimension greater than $1$ was proven by Yano \cite{Y}. Recently, Yano's result was strengthened to show that the generic homeomorphism or continuous map in most settings has an ergodic measure of infinite entropy \cite{CT}. It would be interesting to know if this result holds also in $C_{\lambda}(M)$.

On the other hand the question about generic value of metric entropy is unclear.
Let $PAM_\lambda(M)$ denote the set of piecewise affine Markov maps that preserve Lebesgue measure. 
In \cite[Proposition 24]{BT} the following theorem is proven for $I$, the proof is analogous in the circle case:

\begin{thm}
For every $c\in (0, \infty)$ the set $PAM_{\lambda}(M)_{entr=c}$ is dense in $C_{\lambda}(M)$.
\end{thm}

Furthermore, the following theorem is proven in \cite[Proposition 25]{BT} for $C_{\lambda}(I)$ and can be analogously done with the help of Lemma \ref{PAMdense} for the circle case:

\begin{thm}
The set of maps from $C_{\lambda}(M)$ having metric entropy $\infty$ is dense in $C_{\lambda}(M)$.
\end{thm}

\begin{question}
Does there exist a generic value of metric entropy for maps from $C_{\lambda}(M)$? If such value indeed exists, what is it? If there is no such value, are all non-zero values attained by the metric entropy on every generic set?
\end{question}

\section{Smoothness versus nowhere differentiability}\label{sec:smoothness}

\subsection{Connections between topological and measure-theoretic properties}

In \cite{BCOT2} an effort has been made to understand natural conditions when topological dynamical properties imply the corresponding measure-theoretic properties in $C_{\lambda}(M)$ (and vice versa). The assumptions in the next two theorems come directly from the articles of Li and Yorke \cite{LY} and Bowen \cite{Bowen} respectively. The first result was proven in \cite[Theorem 3]{BCOT2}.

\begin{thm}\label{thm:transitiveergodic}
	Let $f\in C_{\lambda}(M)$ be a piecewise $C^2$ map with a slope strictly greater than $1$.
	Then $f$ is transitive if and only if $(f,M,\lambda)$ is ergodic.
\end{thm}

It is well known that measure-theoretic exactness implies measure-theoretic strong mixing, and since $\lambda$ is positive on open sets this furthermore implies topological mixing.  The first three points of the next result were proven in \cite[Theorem 2]{BCOT2},  which shows that these are equivalent in a smooth enough one dimensional settings. They are also an important ingredient of the proof of Theorem \ref{thm:3}.

\begin{thm}\label{thm:mixingexact}
	Let $f\in C_{\lambda}(M)$ be a piecewise $C^2$ map.
 Then the following conditions are equivalent:
 \begin{enumerate}
 \item\label{thmix:1} $f$ is topologically mixing map, 
 \item $f$ is strongly mixing,
 \item\label{thmix:3} $(f,M,\lambda)$ is measure-theoretically exact,
 \item\label{thmix:4} $f$ is leo.
 \end{enumerate}
\end{thm}

To see that \eqref{thmix:4} is equivalent to \eqref{thmix:1}--\eqref{thmix:3} note that if $f\in C_{\lambda}(M)$ and is piecewise $C^2$ then it can have only a finite number of turning points (which in fact, must be endpoints of pieces on which the map is $C^2$). 
But by \cite{HKO}, a mixing interval map $f$ which is not leo has infinitely many turning points.

Theorems~\ref{thm:transitiveergodic} and \ref{thm:mixingexact} are in strong contrast to Theorem~\ref{t-PP}, which 
displays a difference between piecewise smooth and non-differentiable settings.

\subsection{Nowhere differentiability, knot points and topological entropy}\label{subsec:nowherediff}

In \cite{BS1} Bobok and Soukenka studied continuous piecewise affine interval maps with countably many pieces of monotonicity that preserve the Lebesgue measure. 
By taking limits of such maps they
proved the following theorem:

\begin{thm}\label{thm:BobokSoukenka} There exists a map $g\in C_{\lambda}(I)$ such that:
\begin{enumerate}
\item $g$ is nowhere monotone,
\item knot points of g are dense in $I$ and for a dense $G_{\delta}$ set $Z$ of $z$’s, the
set $g^{-1}({z})$ is infinite; 
\item topological entropy $h_{\mathrm{top}}(g)\leq \mathrm{log}(2) + \eps$.
\end{enumerate}
\end{thm}

Furthermore, the following two (yet unanswered) questions arose from their study:

\begin{question} Does every continuous nowhere differentiable interval map from $C_{\lambda}(I)$ have infinite topological entropy?
\end{question}

The next question relates topological entropy with knot points (see Subsection~\ref{subsec:knotpoint}).

\begin{question} Does every map from $C_{\lambda}(I)$ with a knot point $\lambda$-a.e.\ have infinite topological entropy?
\end{question}

Bobok and Soukenka continued their study in \cite{BS2} where they studied a special conjugacy class $\mathcal{F}$ of continuous piecewise monotone interval maps with countably many laps (including Lebesgue measure-preserving maps), which are locally eventually onto and all have topological entropy $\log(9)$.
They show that
there exist maps from $\mathcal{F}$ with knot points  in its fixed point $1/2$.

\section{Other topologies}\label{sec:other}

The aspects we discuss above are also interesting with other topologies on the spaces of Lebesgue measure-preserving maps on one-dimensional compact manifolds as well as higher dimensional analogues.

de Faria et.\ al.\ showed that topological entropy is infinite for homeomorphisms in two different settings \cite{FHT1,FHT}. As above, let $\mathbb{M}$ be  compact $d$-dimensional manifold and $\mathcal{H}^1(\mathbb{M})$  be
the space of homeomorphisms which are bi-Lipschitz in all local charts.
Let $\mathcal{H}^1_\alpha$ denote the closure of $\mathcal{H}^1$ with
respect to the $\alpha$-H\"older-Whitney topology.
Their first result is that topological entropy is generically 
infinite in $\mathcal{H}^1_\alpha$ whenever $d \ge 2$ and $0 < \alpha <1$.

For $1 \le p,p^* < \infty$ let $S^{p,p^*}(\mathbb{M})$ denote the space of homeomorphisms on $\mathbb{M}$ which in all local charts are of Sobolev class $W^{1,p}$ and whose inverse is of Sobolev class $W^{1,p^*}$ together with
the $(p, p^*)$-Sobolev-Whitney topology.
Their second result is that topological entropy is generically infinite in $S^{p,p^*}(\mathbb{M})$ when $d \ge 2$ and $d - 1 < p,p^* < \infty$.

In \cite{Hazard} Hazard constructed interesting examples of noninvertible maps with infinite topological entropy in these topologies, however he did not study generic behavior.

These results are the first dynamical genericity results for intermediate smoothness.  Generic values of topological entropy in these topologies has not yet been studied in the volume preserving case. In fact no other dynamical properties have been studied and it would be interesting to understand which of the results of this survey hold in analogous topologies for continuous Lebesgue measure preserving maps as well as simply for the continuous maps.

\section*{Appendix}

Let $\textrm{Rec}(f)$ denote the set of recurrent points of a map $f\in C(M)$. It is proved in \cite{MS} that for graph maps (in particular for $\mathbb{S}^1$) $\overline{\mathrm{Rec}(f)}=\overline{\mathrm{Per}(f)}\cup \mathrm{Rec}(f)$. The following lemma confirms intuition that recurrence without dense periodicity is possible only for irrational rotations of $\mathbb{S}^1$. The following lemma is the crucial step in the proof of Remark~\ref{rem:CP} in the case of $\mathbb{S}^1$.

\begin{lemmaA}\label{lem:rec}
Let $f\in C(\mathbb{S}^1)$.
If $\overline{\mathrm{Rec}(f)}=\mathbb{S}^1$ and $\mathrm{Per}(f)\neq \emptyset$, then $\overline{\mathrm{Per}(f)}=\mathbb{S}^1$.
\end{lemmaA}

\begin{proof}
Assume towards contradiction that there is a nonempty open set $U:=\mathbb{S}^1\setminus \overline{\mathrm{Per}(f)}$.
Take a point $x\in U\cap \mathrm{Rec}(f)$ and its maximal omega limit set $W\supset\omega(x)$. Then by \cite{BlokhBM1} (cf. \cite{Ruette}) $W$ is one of the following four  types: a periodic orbit, a basic set, a solenoidal set or 
 a circumferential set. We claim that none of these can occur, and so we have a contradiction, thus $\overline{\mathrm{Per}(f)}=\mathbb{S}^1$. 

 Recall that $x\in \omega(x)$, so $W\cap U\neq \emptyset$, thus $W$ can not be a periodic orbit. 
 Since orbit of $x$ is infinite, there are nonnegative integers $k<n<m$ such that $f^k(x), f^n(x), f^m(x)\in U$.
 If $W$ is a solenoidal set, then there is a periodic interval $J$ such that each of these three points belongs to a different iterate of $J$. In particular, there is a non-negative integer $s$ such that $f^s(J)\subset U$. But since $J$ is a periodic interval, $f^s(J)\cap \per(f)\neq \emptyset$ which is again impossible. If $W$ is a circumferential set, then there is a a connected set $K\supset W$ and a monotone factor map $\phi \colon K\to \mathbb{S}^1$ semi-conjugating $f|_K$ with an irrational rotation and such that each non-singleton fiber of $\phi$ is a wandering interval. In particular, if $\phi$ is not one-to-one then $\mathbb{S}^1\neq \overline{\textrm{Rec}(f)}$ which is a contradiction. But if $\phi$ is one-to-one then $K$ is homeomorphic to $\mathbb{S}^1$ which is again impossible, since $W=K$ is a subset of $\mathbb{S}^1\setminus \per(f)\neq \mathbb{S}^1$ but interval is not homeomorphic to a circle. The last case is that $W$ is a basic set. By repeating the argument from the circumferential set case, we can find a map $\phi$ which conjugates $f^n|_W$ with a mixing interval map, for some $n$. But a consequence of this conjugacy is that periodic points are dense in $W$, in particular there is a periodic point arbitrarily close in $x$, so also in $U$, which is again a contradiction. 
 \end{proof}

 \section*{Acknowledgements} The authors thank the European Regional Development Fund, project No.~CZ 02.1.01/0.0/0.0/16\_019/0000778 for financing their one week research meeting in Prague where this survey article was finalized.


\begin{thebibliography}{AB}

\bibitem{AHK} E.\ Akin, M.\ Hurley, J.A.\ Kennedy, 
{\em Dynamics of topologically generic homeomorphisms}, 
Mem.\ Amer.\ Math.\ Soc., 2003, {\bf 164}, no.\ 783.

\bibitem{A00} E.\ Akin {\em Stretching the Oxtoby-Ulam Theorem}
Colloquium Math.\ \textbf{84/85}  (2000), 83--94.  

\bibitem{AdRR} Ll.\ Alsedà, M.A.\ del Río, J.A.\ Rodríguez, 
{\em Transitivity and dense periodicity for graph maps}. J.\ Difference Equ.\ Appl.\ 9 (2003), no.\ 6, 577--598

\bibitem{AP} S.\ Alpern, V.S.\ Prasad, {\em Typical dynamics of volume preserving homeomorphisms},
Cambridge Tracts in Mathematics, 139,
Cambridge University Press, Cambridge, 2000, 216 pp.
	
	\bibitem{BeKwMe} S.\ Bezuglyi, J.\ Kwiatkowski, K.\ Medynets, {\em Approximation in ergodic theory, Borel, and Cantor dynamics},
	Algebraic and topological dynamics, 
	Contemp.\ Math., {\bf 385}, Amer.\ Math.\ Soc., Providence, RI (2005), 39--64.

 \bibitem{BC} L.S.\ Block, E.M.\ Coven, {\em Topological conjugacy and transitivity for a class of piecewise
monotone maps of the interval,} Trans.\ Amer.\ Math.\ Soc.\ 300 (1987) 297--306.

\bibitem{BlokhBM1} A.M.\ Blokh, {\em Dynamical systems on one-dimensional branched manifolds.} I. (Russian) ; translated from Teor.\ Funktsii Funktsional.\ Anal.\ i Prilozhen.\ No.\ 46 (1986), 8--18 J.\ Soviet Math.\ 48 (1990), no.\ 5, 500--508

\bibitem{BlokhSurv} A.M.\ Blokh,  {\em Spectral decomposition, periods of cycles and a conjecture of M.\ Misiurewicz for graph maps.} Ergodic theory and related topics, III (Güstrow, 1990), 24--31, Lecture Notes in Math., 1514, Springer, Berlin, 1992

 \bibitem{Blokh} A.M.\ Blokh, {\em The ``spectral'' decomposition for one-dimensional maps.} Dynamics reported, 1–59, Dynam.\ Report.\ Expositions Dynam.\ Systems (N.S.), 4, Springer, Berlin, 1995.
	
	\bibitem{B} J.\ Bobok, {\em On non-differentiable measure-preserving functions}, Real Analysis Exchange {\bf 16}(1) (1991), 119--129.
	
	\bibitem{BCOT} J.\ Bobok, J.\ \v Cin\v c, P.\ Oprocha, S.\ Troubetzkoy, {\em S-limit shadowing is generic for continuous Lebesgue measure-preserving circle maps}, Erg. Theory Dyn. Sys., {\bf 43} (1), (2023) 78--98. 
	
	\bibitem{BCOT1} J.\ Bobok, J.\ \v Cin\v c, P.\ Oprocha, S.\ Troubetzkoy, {\em Periodic points and shadowing property of Lebesgue measure-preserving interval maps}, Nonlinearity {\bf 35} (2022) 2534--2557.
	
	\bibitem{BCOT2} J.\ Bobok, J.\ \v Cin\v c, P.\ Oprocha, S.\ Troubetzkoy, {\em Are generic dynamical properties stable under composition with rotations?}, arXiv:2207.07186, July 2022.
	
	\bibitem{BCOT3} J.\ Bobok, J.\ \v Cin\v c, P.\ Oprocha, S.\ Troubetzkoy, {\em Generic maps with a dense set of periodic points}, in preparation 2022.\
	
	\bibitem{BS1} J.\ Bobok, M.\ Soukenka, {\em Irreducibilty, infinite level sets and small entropy} Real Analysis Exchange {\bf 36}(2), 2010/2011,  449--462.
	
	\bibitem{BS2} J.\ Bobok, M.\ Soukenka, {\em On piecewise affine interval maps with countably many laps.} Disc.\ Cont.\ Dyn.\ Syst.\ {\bf 31}(3), 2011, 753--762.
	
	\bibitem{BT} J.\ Bobok, S.\ Troubetzkoy, {\em Typical properties of interval maps preserving the Lebesgue measure}, Nonlinearity, \textbf{33}  (2020), 6461--6501.
	
	\bibitem{BCO} J.\ Boro\'nski, J.\ \v Cin\v c, P.\ Oprocha, {\em Beyond $0$ and $\infty$: The solution to the Barge entropy conjecture},
	arXiv:2105.11133, May 2021.
	
	\bibitem{Bowen} R.\ Bowen, {\em Bernoulli maps of the interval}, Isr.\ J.\ Math.\ {\bf 28} (1977), 161--168.
	
	\bibitem{Buzzi} J.\ Buzzi, {\em Specification on the interval}, Trans.\ Amer.\ Math.\ Soc.\ \textbf{349} (1997), 2737--2754.
	
	\bibitem{PV} M.\ Carvalho, F.B.\ Rodrigues and P.\ Varandas, {\em Generic homeomorphisms have full metric mean dimension}, Ergod.\ Theory Dyn.\ Syst.\ {\bf 42}(1) (2022), 40--64.

    \bibitem{CT} E.\ Catsigeras, S.\ Troubetzkoy, {\em Ergodic measures with infinite entropy}, arXiv:1901.07221, 2019.
	
	\bibitem{CO} J.\ \v Cin\v c, P.\ Oprocha, {\em Parametrized family of pseudo-arc attractors: Physical measures and prime end rotations}, Proc.\ London Math.\ Soc.\ (3) {\bf 125}, (2022), 318–357.

 \bibitem{COpseudocircle} J.\ \v Cin\v c, P.\ Oprocha, {\em Parameterized family of pseudo-circle attractors}, work in progress.
	
	\bibitem{CoHe80}E.\ M.\ Coven, G.A.\ Hedlund, {\em$\overline{P}=\overline{R}$ for maps of the interval}, Proc.\ Amer.\ Math.\ Soc.\ \textbf{79} (1980), 316--318.
	
	
	\bibitem{ChPr} J.R.\ Choksi, V.S.\ Prasad, {\em Approximation and Baire category theorems in ergodic theory},
	Measure theory and its applications (Sherbrooke, Que.\, 1982), 94--113, Lecture Notes in Math., {\bf 1033}, Springer, Berlin, 1983.

 \bibitem{FHT1} E.\ de Faria, Edson, P.\ Hazard, C.\ Tresser, {\em Infinite entropy is generic in Hölder and Sobolev spaces.} C.R.\ Math.\ Acad.\ Sci.\ Paris {\bf 355} (2017), no.\ 11, 1185--1189.
 

\bibitem{FHT} E.\ de Faria, P.\ Hazard, C.\ Tresser, {\em Genericity of infinite entropy for maps with low regularity}, Ann.\ Sc.\ Norm.\ Super.\ Pisa Cl.\ Sci.\ (5) 22 (2021), no.\ 2, 601--664.
	
	
\bibitem{PAG1}
P.-A.\ Guih\'eneuf,
{\em Propriétés dynamiques génériques des homéomorphismes conservatifs}
Ensaios Matematicos [Mathematical Surveys], {\bf 22}, Sociedade Brasileira de Matemática, Rio de Janeiro, 115 pp, 2012.

\bibitem{PAG} P.-A.\  Guih\'eneuf, {\em Ensemble des points p\'eriodiques d'un hom\'eomorphisme conservatif g\'en\'erique}  \url{https://webusers.imj-prg.fr/~pierre-antoine.guiheneuf/Fichiers/PerGeneHom.pdf}
	
	\bibitem{GuLe18} P.-A.\ Guih\'eneuf, T.\ Lefeuvre,  {\em On the genericity of the shadowing property for conservative homeomorphisms}, Proc.\ Am.\ Math.\ Soc.\ {\bf 146} (2018), 4225--4237.
	
	\bibitem{Ha44} P.R.\ Halmos,
	{\em In general a measure-preserving transformation is mixing},
	Ann.\ of Math.\ {\bf 45}(2), (1944), 786--792.
	
	\bibitem{Ha44.1} P.R.\ Halmos,
	{\em Approximation theories for measure-preserving transformations},
	Trans.\ Amer.\ Math.\ Soc.\ {\bf 55} (1944), 1--18.

    \bibitem{HKO} G.\ Hara\'nczyk, D.\ Kwietniak, P.\ Oprocha, {\em Topological structure and entropy of mixing graph maps}, Ergodic Theory Dynam.\ Systems 34 (2014), no.\ 5, 1587--1614.

    \bibitem{Hazard} P.\ Hazard, {\em Maps in dimension one with infinite entropy}. Ark.\ Mat.\ 58 (2020), no.\ 1, 95--119.

    \bibitem{MS} J.-H.\ Mai and S.\ Shao, {\em $\overline R=R\cup\overline P$ for graph maps.} J. Math. Anal. Appl. \textbf{350} (2009), no. 1, 9--11.

	
	\bibitem{KaSt} A.\ Katok, A.\ Stepin, {\em Approximations in ergodic theory}, Uspehi Mat.\
	Nauk \textbf{22} (1967), 81--106.

	
	\bibitem{KoMaOpKu} P.\ Koscielniak, M.\ Mazur, P.\ Oprocha,  L.\ Kubica, {\em Shadowing is Generic on Various One-Dimensional Continua with a Special Geometric Structure}, J.\ Geom.\ Anal.\ {\bf 30}(2) (2020), 1836--1864.

	
	\bibitem{Lewis}
	W.\ Lewis, {\em The pseudo-arc.} Continuum theory and dynamical systems (Arcata, CA, 1989), 103–123, Contemp. Math., {\bf 117}, Amer.\ Math.\ Soc., Providence, RI, 1991.
	
	\bibitem{LY} T.Y.\ Li, J.A.\ Yorke, {\em Ergodic transformation from an interval into itself}, Trans. Amer.\ Math.\ Soc.\ {\bf 235}(1978), 183--192.
	
	\bibitem{MazOpr} M.\ Mazur, P.\ Oprocha, \ {\em S-limit shadowing is $\mathcal{C}^0$-dense}, J. Math.\ Anal. Appl.\ {\bf 408}(2) (2013), 465--475.
	
	\bibitem{Mizera} I.\ Mizera, {\em Generic properties of one-dimensional dynamical systems}, Ergodic theory and related topics, III (Gu\"ustrow, 1990), Lecture Notes in Math. {\bf 1514}  (1992), Springer, Berlin, 163--173.
	
	\bibitem{Med} J.\ Meddaugh, {\em On genericity of shadowing in one-dimensional continua}, Fund.\ Math.\ {\bf 255} (2021), 1--18.
	
	\bibitem{MT} P.\ Minc, W.\ R.\ R.\ Transue, \emph{A transitive map on $I$ whose inverse limit is the pseudo-arc}, Proc. Amer.\ Math.\ Soc.\ {\bf 111}, no. 4 (1991), 1165--1170.

\bibitem{ChrisM} C. Mouron, \textit{Entropy of shift maps of the pseudo-arc.} Topology Appl. \textbf{159} (2012), no. 1, 34--39. 
 
	\bibitem{OxUl} J.\ Oxtoby, S.\ Ulam, {\em Measure-preserving homeomorphisms and metric transitivity}, Ann.\ Math.\ {\bf 42} (1941), 874--920.
	
	
	\bibitem{Pet} K.\ Petersen, {\em Measure Preserving Systems}, Encyclopedia of Complexity and Systems Science, 5450--5466.
	
	\bibitem{Pi} S.Yu.\  Pilyugin, O.B.\  Plamenevskaya,  {\em Shadowing is generic},Topology Appl.\  {\bf 97}(3) (1999), 253--266.
	

\bibitem{Ruette} S. Ruette, L. Snoha, \textit{For graph maps, one scrambled pair implies Li-Yorke chaos.} Proc. Amer. Math. Soc. 142 (2014), no. 6, 2087--2100.
 
	\bibitem{Ro48} V.\ Rohlin, {\em A ``general'' measure-preserving transformation is not mixing},
	Doklady Akad.\ Nauk SSSR (N.S.) {\bf 60} (1948), 349--351.

\bibitem{Sa32} S.\ Saks,  {\em On the functions of Besicovitch in
    the space of continuous functions},
Fundamenta Mathematicae  19 (1932), 211--219.
 
    
  \bibitem{Sa37}  S. Saks, On the functions of Besicovitch in the space of continuous functions, Fundamenta Mathematicae {\bf 19} (1932), 211--219.
	
	\bibitem{SW95} J.\ Schmeling, R.\ Winkler, {\em Typical dimension of the graph of certain functions}, Monatsh. Math. {\bf 119} (1995), 303-–320.
	
\bibitem{Tr} S.\ Troubetzkoy, {\em Approximation and billiards}, Dynamical systems and Diophantine approximation, 173--185, Sémin. Congr., 19, Soc. Math. France, Paris, 2009.

\bibitem{Walters} P.\ Walters, An Introduction to Ergodic Theory (New York: Springer), 1982.

\bibitem{Y} K.\ Yano, {\em A remark on the topological entropy of homeomorphisms}, Inv.\ Math. {\bf 59} (1980), 215--220.


\end{thebibliography}
\end{document}